\documentclass[12pt,a4paper,reqno]{amsart}

\usepackage[margin=1in,footskip=0.25in]{geometry}
\usepackage{amsmath, amsthm, amssymb}
\usepackage{graphicx}

\usepackage{enumitem}
\usepackage[verbose]{hyperref}

\newtheorem{theorem}{Theorem}

\date{}

\begin{document}

\title[Crossing limit cycles for a family of isochronous centers]
{Emergence of Multiple Crossing Limit Cycles in Planar Piecewise Systems with Isochronous Centers and Nonsmooth Switching Manifolds}

\author{Sonia Isabel Renteria Alva$^1$}

\address{$^1$ Instituto de Matemática Pura e Aplicada, Estrada Dona Castorina 110,
Jardim Botânico, 22460-320,
 Rio de Janeiro, Brazil}
\email{sonia.alva@impa.br}

\author{Pedro Iván Suárez Navarro$^2$}
\address{$^2$ Instituto de Matemática Pura e Aplicada, Estrada Dona Castorina 110,
Jardim Botânico, 22460-320,
 Rio de Janeiro, Brazil}
\email{ivan.suarez@impa.br}



\begin{abstract}
Discontinuous piecewise differential systems exhibit dynamical behaviors with no counterpart in smooth systems, particularly in the presence of nonsmooth switching structures. In this work, we extend previous results for systems separated by a straight line to the case where the switching manifold is a nonregular curve, showing that the loss of regularity significantly increases the algebraic complexity of the closing conditions defining crossing limit cycles. As a consequence, we derive explicit upper bounds for the number of crossing limit cycles in planar systems formed by a linear Hamiltonian saddle and quadratic isochronous centers, and construct explicit examples exhibiting four crossing limit cycles in each case, thereby providing sharp constructive lower bounds. While the upper bounds follow from classical algebraic arguments, the realization of multiple crossing limit cycles requires solving nonlinear systems of high degree and remains highly nontrivial. These results highlight how nonsmooth switching manifolds enhance dynamical complexity and promote multistability in discontinuous piecewise systems.
\end{abstract}

\maketitle

\section{Introduction and statement of the main results}
Discontinuous piecewise differential systems have attracted considerable attention in recent years due to their ability to model complex phenomena arising in engineering, physics, and applied sciences. Within the Filippov framework \cite{filippov2013differential}, such systems allow the coexistence of distinct vector fields separated by a switching manifold, giving rise to a variety of nonsmooth dynamical behaviors including crossing, sliding, and grazing phenomena. The interplay between geometry and dynamics in these systems often leads to effects that have no counterpart in smooth differential equations \cite{bernardo2008piecewise,simpson2010bifurcations}.

A central topic in the qualitative theory of dynamical systems is the study of limit cycles and their multiplicity, which is closely related to the second part of Hilbert's 16th problem. In the nonsmooth setting, this problem has evolved into the analysis of crossing limit cycles in piecewise systems, where periodic orbits intersect the discontinuity set transversally. Significant progress has been made in recent years, particularly for systems composed of linear Hamiltonian saddles and isochronous centers \cite{esteban202116th,llibre202516th,llibre2022crossing,llibre2021limit}. These works show that even low-degree systems can generate multiple crossing limit cycles, highlighting the intrinsic richness of nonsmooth dynamics.

An important aspect influencing this behavior is the geometry of the switching manifold. While many classical results assume regular (smooth or straight-line) separation sets, recent studies have shown that nonregular switching manifolds can dramatically increase the dynamical complexity of the system \cite{llibre2023extended,huan2019limit,baymout2024limit,he2024limit}. In particular, geometric irregularities introduce additional nonlinear constraints in the matching conditions that define periodic orbits, leading to higher algebraic complexity and new mechanisms for the generation of limit cycles.

From an analytical perspective, the detection of crossing limit cycles is typically reduced to solving nonlinear systems derived from the closing conditions between trajectories on each side of the discontinuity. This reduction naturally leads to polynomial equations of increasing degree, for which Bézout's arguments provide effective upper bounds on the maximum number of solutions. However, while such bounds follow from classical algebraic considerations, the explicit realization of multiple crossing limit cycles remains a highly nontrivial problem, requiring delicate constructions and careful control of parameters.

 We consider the extension of the 16th Hilbert's problem on planar discontinuous piecewise differential systems separated by a nonregular switching manifold and formed by a linear Hamiltonian saddle and quadratic isochronous centers. Our main goal is to understand how the loss of regularity in the discontinuity set affects the structure of crossing limit cycles. We show that geometric irregularity acts as a mechanism that enhances algebraic complexity and promotes the emergence of multiple crossing limit cycles. Moreover, we construct explicit examples exhibiting several crossing limit cycles, providing constructive evidence of multistability in this class of systems, in line with recent advances on realizability \cite{alva2025crossing,berbache2025lower,llibre2026limit}.

These results contribute to the growing body of literature indicating that nonsmooth systems, especially those with irregular switching structures, constitute a natural setting for the emergence of complex dynamics, where geometry and algebra interact in a nontrivial way to produce rich and sometimes unexpected behaviors.

In this paper we study planar piecewise differential systems formed by 
\begin{align}\label{eqmainOrig}
   (\dot x, \dot y)=\begin{cases}
(X^{+}(x,y), Y^{+}(x,y)), & \text{if } (x,y)\in \Sigma^{+},\\
(X^{-}(x,y), Y^{-}(x,y)), & \text{if } (x,y)\in \Sigma^{-},
\end{cases} 
\end{align}

where the separation curve is the nonregular line
$\Sigma = \Sigma^{+} \cap \Sigma^{-}$ such that
$$\Sigma^{+}=\lbrace (x,y): x\geq 0 \text{ and } y\geq 0 \rbrace, \quad  \Sigma^{-}=\lbrace (x,y): x\leq 0 \rbrace \cup \lbrace (x,y): x\geq 0 \text{ and } y\leq 0 \rbrace. $$

According to \cite{filippov2013differential}, a point $p$ on the discontinuity set is a \textit{crossing point} if $X^{-}(p)X^{+}(p) > 0$. A periodic solution of system~\eqref{eqmainOrig} is of crossing type if it intersects the discontinuity curve exactly twice at crossing points, and it is called a \textit{crossing limit cycle} if it is isolated among such solutions. For simplicity, we use the term \emph{limit cycle} to refer to crossing limit cycles.

In this paper, we extend the results obtained in~\cite{llibre202516th} for discontinuous piecewise differential systems separated by a straight line to the case where the discontinuity set is a nonregular curve, while preserving the same isochronous centers. In the straight-line setting, the authors showed that the maximum number of crossing limit cycles arising from the combination of a linear Hamiltonian saddle in $x>0$ with any of the four quadratic isochronous centers $\mathtt{Q}_{i}$ in $x<0$ is two. They also provided examples exhibiting one crossing limit cycle, leaving open the question of whether the upper bound can actually be attained.

In contrast, the present framework reveals a richer and more intricate dynamical scenario. The introduction of a nonregular discontinuity curve leads to additional technical difficulties and a more complex geometric structure. Despite these challenges, we are able to construct explicit systems exhibiting at least four crossing limit cycles, highlighting a significant increase in dynamical complexity compared to the classical straight-line case.
 
In this work, the analysis is restricted to the region $\Sigma^{-}$, where quadratic polynomial differential systems of the form
\begin{equation}\label{sistemacubico}
\begin{aligned}
\dot{x} &= -y + a_{20}x^2 + a_{11}xy + a_{02}y^2,\\
\dot{y} &= \phantom{-}x + b_{20}x^2 + b_{11}xy + b_{02}y^2,
\end{aligned}
\end{equation}
are considered.

Loud \cite{loud1964behavior}  classified the quadratic systems of the form ~\eqref{sistemacubico} that possess an isochronous center at the origin. More precisely, such a system has an isochronous center at the origin if and only if it can be transformed, by means of a linear change of coordinates and a rescaling of time, into one of the following four differential systems:

\begin{equation}\label{sistemas1s2s3s4}
\begin{aligned}
&(\mathtt{Q}_1): \;
\left\{
\begin{aligned}
\dot{x} &= -y + x^2 - y^2, \\
\dot{y} &= x(1 + 2y),
\end{aligned}
\right.
\qquad
(\mathtt{Q}_2): \;
\left\{
\begin{aligned}
\dot{x} &= -y + x^2, \\
\dot{y} &= x(1+y),
\end{aligned}
\right.
\\[0.3cm]
&(\mathtt{Q}_3): \;
\left\{
\begin{aligned}
\dot{x} &= -y -\frac{4}{3}x^2, \\
\dot{y} &= x\!\left(1-\frac{16}{3}y\right),
\end{aligned}
\right.
\qquad
(\mathtt{Q}_4): \;
\left\{
\begin{aligned}
\dot{x} &= -y +\frac{16}{3}x^2-\frac{4}{3}y^2, \\
\dot{y} &= x\!\left(1+\frac{8}{3}y\right).
\end{aligned}
\right.
\end{aligned}
\end{equation}

The first integrals associated with the systems~\eqref{sistemas1s2s3s4} were obtained in \cite{chavarriga1999survey} and are given, respectively, by
\begin{equation}\label{intprimeiras1s2s3s4}
\begin{aligned}
&(\mathtt{Q}_1): \quad \widetilde{H}_1(x,y)=\frac{x^2+y^2}{1+2y}, \qquad
(\mathtt{Q}_2): \quad \widetilde{H}_2(x,y)=\frac{x^2+y^2}{(1+y)^2},\\[0.2cm]
&(\mathtt{Q}_3): \quad \widetilde{H}_3(x,y)=\frac{9(x^2+y^2)-24x^2y+16x^4}{-3+16y}, \quad
(\mathtt{Q}_4): \quad \widetilde{H}_4(x,y)=\frac{9(x^2+y^2)+24y^3+16y^4}{(3+8y)^4}.
\end{aligned}
\end{equation}

In contrast, in the region $\Sigma^{+}$ we consider a linear Hamiltonian saddle. As shown in \cite{llibre2021limit} after an affine transformation and a rescaling of the independent variable, any linear Hamiltonian saddle can be written in the form
\begin{equation}\label{sistemalinear}
(\mathtt{L}_S):\;
\begin{aligned}
\dot{x} &= -A x - \delta y + B,\\
\dot{y} &= \mu x + A y + C,
\end{aligned}
\end{equation}
where $\mu$, $A$, $\delta$, $B$, and $C$ are real parameters satisfying $A^{2}-\delta\mu>0$. 
Moreover, system~\eqref{sistemalinear} admits the first integral
\begin{equation}\label{intprimeiralinear}
H_S(x,y) = -\frac{\mu}{2}x^{2} - Axy - \frac{\delta}{2}y^{2} - Cx + By.
\end{equation}

Our main result establishes the maximum number of crossing limit cycles that may arise in discontinuous piecewise differential systems of the form~\eqref{eqmainOrig}, where in the region $\Sigma^{+}$ the dynamics is governed by a linear Hamiltonian saddle $(\mathtt{L}_s)$, and in the region $\Sigma^{-}$ the system corresponds, after an arbitrary affine change of variables, to one of the quadratic isochronous differential systems $(\mathtt{Q}_1)$, $(\mathtt{Q}_2)$, $(\mathtt{Q}_3)$, or $(\mathtt{Q}_4)$. This result highlights how the geometry of the discontinuity curve directly influences the number of crossing limit cycles, producing a substantial increase compared to the classical straight-line case. 

In this work, crossing limit cycles intersect the discontinuity curve $\Sigma$ at exactly two points. The case of more than two intersections remains open.

\begin{theorem}\label{Teor-Principal}
Consider the class of discontinuous piecewise differential systems separated by a nonregular line and composed of two differential systems which, after an affine change of variables, belong to the classes $(\mathtt{L}_s)$ and $(\mathtt{Q}_i)$, with $i\in\{1,2,3,4\}$. Then the following statements hold:
\begin{itemize}
\item[(i)] For systems of type $(\mathtt{L}_s)$-$(\mathtt{Q}_1)$, the maximum number of limit cycles is five. Moreover, there exist systems of this type exhibiting exactly four limit cycles (see Fig.~\ref{fig-Hs-H1}).
\item[(ii)] For systems of type $(\mathtt{L}_s)$-$(\mathtt{Q}_2)$, the maximum number of limit cycles is seven. Moreover, there exist systems of this type with exactly four limit cycles (see Fig.~\ref{fig-Hs-H2}).
\item[(iii)] For systems of type $(\mathtt{L}_s)$-$(\mathtt{Q}_3)$, the maximum number of limit cycles is nine. Furthermore, there exist systems of this type possessing exactly four limit cycles (see Fig.~\ref{fig-Hs-H3}).
\item[(iv)] For systems of type $(\mathtt{L}_s)$-$(\mathtt{Q}_4)$, the maximum number of limit cycles is eleven. In addition, there exist systems of this type with exactly four cycles (see Fig.~\ref{fig-Hs-H4}).
\end{itemize}
\end{theorem}
The paper is organized as follows. In Section \ref{sect:02} we describe the quadratic isochronous centers under affine transformations. Section \ref{sec:proof} is devoted to the proof of the main result. Explicit constructions exhibiting multiple crossing limit cycles are also provided, illustrating the sharpness of the lower bounds.
\section{The quadratic isochronous differential systems ($\mathtt{Q}_1$), ($\mathtt{Q}_2$), ($\mathtt{Q}_3$) and ($\mathtt{Q}_4$) after an affine change of variables}
\label{sect:02}
In this section, we give the expression of the quadratic
isochronous differential systems ($\mathtt{Q}_1$), ($\mathtt{Q}_2$), ($\mathtt{Q}_3$) and ($\mathtt{Q}_4$) and
their first integrals after the general affine change of
variables $(x, y) \rightarrow (a_1 x +b_1 y + c_1, \alpha_1 x + \beta_1 y + \gamma_1)$,
with $b_1 \alpha_1 - a_1 \beta_1\neq 0$.
Thus, after this affine change of variables the differential system ($\mathtt{Q}_1$) becomes 
\begin{equation}\label{sysS1}
\begin{aligned}
\dot{x} &= \frac{1}{b_1 \alpha_1 - a_1 \beta_1}(b_1 c_1 - c_1^2 \beta_1+2 b_1 c_1 \gamma_1 + \beta_1 \gamma_1 + \beta_1 \gamma_1^2
+ (a_1 b_1 + 2 b_1 c_1 \alpha_1 - 2 a_1 c_1 \beta_1\\ 
&\quad + \alpha_1 \beta_1
 +  2 a_1 b_1 \gamma_1  + 2 \alpha_1 \beta_1 \gamma_1)x
+ (b_1^2 + \beta_1^2) (1 + 2 \gamma_1)y
+ (2 a_1 b_1 \alpha_1 - a_1^2 \beta_1 + \alpha_1^2 \beta_1)x^{2}
\\
&\quad +  2\alpha_1 (b_1^2 + \beta_1^2)xy
+ \beta_1 (b_1^2 + \beta_1^2)y^{2}),
\\[0.4em]
\dot{y} &=\frac{1}{b_1 \alpha_1 - 
  a_1 \beta_1} (-a_1 c_1 + c_1^2 \alpha_1  - 
   2 a_1 c_1 \gamma_1 - \alpha_1 \gamma_1 - \alpha_1 \gamma_1^2+
   (-a_1^2 \alpha_1 - \alpha_1^3) x^2 \\
   &\quad - 
   2 (a_1^2 \beta_1 + \alpha_1^2 \beta_1)x y  + 
   (b_1^2 \alpha_1 - 2 a_1 b_1 \beta_1 - \alpha_1 \beta_1^2) y^2 - 
    (a_1^2 + \alpha_1^2) (1 + 2 \gamma_1)x\\
   &\quad+ 
    (-a_1 b_1 + 2 b_1 c_1 \alpha_1 - 
      2 a_1 c_1 \beta_1 - \alpha_1 \beta_1 - 2 a_1 b_1 \gamma_1 - 
      2 \alpha_1 \beta_1 \gamma_1)y),
\end{aligned}
\end{equation}
with the first integral
\[ H_1(x,y) =\frac{(c_1 + a_1 x + b_1 y)^2 + (\alpha_1 x + 
   \beta_1 y + \gamma_1)^2}{1 + 
 2 (\alpha_1 x + \beta_1 y + \gamma_1)}.
\]
The differential system ($\mathtt{Q}_2$) becomes
\begin{equation}\label{sysS2}
\begin{aligned}
\dot{x} &=\frac{1}{b_1 \alpha_1 - 
 a_1 \beta_1}(b_1 c_1 - c_1^2 \beta_1  + 
  b_1 c_1 \gamma_1 + \beta_1 \gamma_1  + 
   (a_1 b_1 \alpha_1 - a_1^2 \beta_1) x^2 + 
  (b_1^2 \alpha_1 - a_1 b_1 \beta_1)x y \\
  &\quad + 
   (a_1 b_1 + b_1 c_1 \alpha_1 - 2 a_1 c_1 \beta_1 + \alpha_1 \beta_1 +
      a_1 b_1 \gamma_1)x + 
  (b_1^2 - b_1 c_1 \beta_1 + \beta_1^2 + b_1^2 \gamma_1) y),
\\[0.4em]
\dot{y} &= \frac{1}{b_1 \alpha_1 - 
 a_1 \beta_1}(-a_1 c_1 + c_1^2 \alpha_1  - 
  a_1 c_1 \gamma_1 - \alpha_1 \gamma_1 + 
  (a_1 b_1 \alpha_1 - a_1^2 \beta_1)x y  + 
  (b_1^2 \alpha_1 - a_1 b_1 \beta_1)y^2  \\
  &\quad+ 
 (-a_1^2 + a_1 c_1 \alpha_1 - \alpha_1^2 - a_1^2 \gamma_1) x  + 
   (-a_1 b_1 + 2 b_1 c_1 \alpha_1 - 
     a_1 c_1 \beta_1 - \alpha_1 \beta_1 - a_1 b_1 \gamma_1)y),
\end{aligned}
\end{equation}
with the first integral
\[
H_2(x,y) =\frac{(c_1 + a_1 x + b_1 y)^2 + (x \alpha_1 + y \beta_1 + \gamma_1)^2}{(1 +
   x \alpha_1 + y \beta_1 + \gamma_1)^2}.
\]
The differential system ($\mathtt{Q}_3$) becomes
\begin{equation}\label{sysS3}
\begin{aligned}
\dot{x} &= 
\frac{1}{3(b_1 \alpha_1 - 
  a_1 \beta_1)} (3 b_1 c_1 + 4 c_1^2 \beta_1- 16 b_1 c_1 \gamma_1 +    3 \beta_1 \gamma_1 - 12 \beta_1 b_1^2 y^2  +    4 x^2 (-4 a_1 b_1 \alpha_1 + a_1^2 \beta_1)\\
  &\quad -    8 x y (2 b_1^2 \alpha_1 + a_1 b_1 \beta_1)  + 
   x (3 a_1 b_1 - 16 b_1 c_1 \alpha_1 + 8 a_1 c_1 \beta_1  + 
      3 \alpha_1 \beta_1 - 16 a_1 b_1 \gamma_1) \\
   &\quad + 
   y (3 b_1^2 - 8 b_1 c_1 \beta_1 + 3 \beta_1^2 - 16 b_1^2 \gamma_1)),
\\[0.4em]
\dot{y} &= \frac{1}{3(b_1 \alpha_1 - 
  a_1 \beta_1)} (-3 a_1 c_1 - 4 c_1^2 \alpha_1  + 16 a_1 c_1 \gamma_1 - 
   3 \alpha_1 \gamma_1  + 
   12 a_1^2 x^2 \alpha_1 + 
   8 x y (a_1 b_1 \alpha_1 + 2 a_1^2 \beta_1)
    \\
   &\quad- 
   4 (b_1^2 \alpha_1 - 4 a_1 b_1 \beta_1)y^2  + 
  (-3 a_1^2 + 8 a_1 c_1 \alpha_1 - 3 \alpha_1^2 + 
      16 a_1^2 \gamma_1) x  + 
  (-3 a_1 b_1 - 8 b_1 c_1 \alpha_1 \\
   &\quad + 16 a_1 c_1 \beta_1 -  3 \alpha_1 \beta_1 + 16 a_1 b_1 \gamma_1) y ),
\end{aligned}
\end{equation}
with the first integral
\begin{equation*}
\begin{aligned}
H_3(x,y)
&=\frac{1}{-3 + 
   16 (x \alpha_1 + y \beta_1 + \gamma_1)}(16 (c_1 + a_1 x + b_1 y)^4 - 
   24 (c_1 + a_1 x + b_1 y)^2 (x \alpha_1 + y \beta_1 \\
   &\quad + \gamma_1) + 
   9 ((c_1 + a_1 x + b_1 y)^2  + (x \alpha_1 + y \beta_1 + \gamma_1)^2)).
\end{aligned}
\end{equation*}
The differential system ($\mathtt{Q}_4$) becomes

\begin{equation}\label{sysS4}
\begin{aligned}
\dot{x} &= \frac{1}{3(b_1 \alpha_1 - 
  a_1 \beta_1)} (3 b_1 c_1 - 16 c_1^2 +\beta_1+ 8 b_1 c_1 \gamma_1 + 
  3 \beta_1 \gamma_1 + 4 \beta_1 \gamma_1^2 - 
  4 (-2 a_1 b_1 \alpha_1 + 
     4 a_1^2 \beta_1  \\
  &\quad - \alpha_1^2 \beta_1)x^2
  +   8 (b_1^2 \alpha_1 - 3 a_1 b_1 \beta_1  + \alpha_1 \beta_1^2)x y  - 
  4 y^2 (2 b_1^2 \beta_1 - \beta_1^3)  + 
(3 a_1 b_1 + 8 b_1 c_1 \alpha_1 - 32 a_1 c_1 \beta_1  \\
  &\quad  +  3 \alpha_1 \beta_1  + 8 a_1 b_1 \gamma_1  \quad +      8 \alpha_1 \beta_1 \gamma_1) x  + 
 (3 b_1^2 - 24 b_1 c_1 \beta_1 + 3 \beta_1^2 + 8 b_1^2 \gamma_1 + 
     8 \beta_1^2 \gamma_1) y ),
\\[0.4em]
\dot{y} &= \frac{1}{3(b_1 \alpha_1 - 
  a_1 \beta_1)}(-3 a_1 c_1 + 16 c_1^2 \alpha_1- 8 a_1 c_1 \gamma_1 - 
  3 \alpha_1 \gamma_1 - 4 \alpha_1 \gamma_1^2 + 
  4 (2 a_1^2 \alpha_1 - \alpha_1^3)x^2 \\
  &\quad - 
  8 (-3 a_1 b_1 \alpha_1 + a_1^2 \beta_1 + \alpha_1^2 \beta_1)x y + 
  4  (4 b_1^2 \alpha_1 - 
     2 a_1 b_1 \beta_1 - \alpha_1 \beta_1^2)y^2 + 
   (-3 a_1^2 + 24 a_1 c_1 \alpha_1 \\
  &\quad  - 3 \alpha_1^2 - 
     8 a_1^2 \gamma_1 - 8 \alpha_1^2 \gamma_1)x + 
  (-3 a_1 b_1 + 32 b_1 c_1 \alpha_1 - 8 a_1 c_1 \beta_1 - 
     3 \alpha_1 \beta_1 - 8 a_1 b_1 \gamma_1 - 
     8 \alpha_1 \beta_1 \gamma_1) y),
\end{aligned}
\end{equation}
with the first integral
\[
\begin{aligned}
H_4(x,y)
&=\frac{1}{(3 + 
   8 (x \alpha_1 + y \beta_1 + \gamma_1))^4}(24 (x \alpha_1 + y \beta_1 + \gamma_1)^3 + 
   16 (x \alpha_1 + y \beta_1 + \gamma_1)^4\\
   &\quad + 
   9 ((c_1 + a_1 x + b_1 y)^2 + (x \alpha_1 + 
        y \beta_1 + \gamma_1)^2)).
\end{aligned}
\]

\section{Proof of Theorem 1}
\label{sec:proof}

\begin{proof}
The proof is based on reducing the existence of crossing limit cycles to the solution of a polynomial system obtained from the matching conditions of the first integrals, and then estimating the number of solutions via algebraic arguments.

We consider a discontinuous piecewise differential system of type $(\mathrm{Ls})$–$(Q_i)$, with $i \in \{1,2,3,4\}$. Let $H_S(x,y)$ and $H_i(x,y)$ denote the corresponding first integrals in $\Sigma^+$ and $\Sigma^-$, respectively.

Assume that the system admits a crossing periodic orbit intersecting the discontinuity set $\Sigma$ at two points of the form $(x,0)$ and $(0,y)$, with $x>0$ and $y>0$. Then these points must satisfy the closing conditions
\begin{equation}\label{crossingEq}
H_S(x,0) = H_S(0,y), \qquad H_i(x,0) = H_i(0,y).
\end{equation}

The first equation can be written as a quadratic polynomial
\[
P_S(x,y) = -2Cx - 2By + \delta y^2 - \mu x^2 = 0.
\]

On the other hand, for each $i \in \{1,2,3,4\}$, the second equation can be expressed as
\[
H_i(x,0) - H_i(0,y) = \frac{P_i(x,y)}{D_i(x,y)} = 0,
\]
where $P_i(x,y)$ is a polynomial and $D_i(x,y) \neq 0$ in the region under consideration. Hence, the problem reduces to solving the polynomial system
\[
P_S(x,y) = 0, \qquad P_i(x,y) = 0.
\]
Let $d_i = \deg(P_i)$. By Bézout's theorem \cite{coolidge2004treatise}, the number of isolated solutions of this system in $\mathbb{C}^2$ is at most $2d_i$. Consequently, the number of real solutions is also bounded by $2d_i$.

Observe that $(0,0)$ is always a solution, but it does not correspond to a crossing limit cycle. Moreover, each crossing limit cycle is uniquely associated with a solution $(x,y)$ satisfying $x>0$ and $y>0$, since it determines the two intersection points of the orbit with $\Sigma$.

Therefore, the number of crossing limit cycles is bounded by the number of admissible solutions $(x,y)$ with $x>0$ and $y>0$, which is at most $2d_i - 1$.

A direct computation shows that $d_1 = 3$, $d_2 = 4$, $d_3 = 5$, and $d_4 = 6$. Hence, the corresponding upper bounds are
 $5$, $7$, $9$, and  $11$, respectively. The explicit expressions of the polynomials $P_i(x,y)$ are given in Appendix~\ref{appendixA}.

This establishes the upper bounds in statements (i)--(iv).

\medskip
It remains to provide explicit constructions showing that each class admits at least four crossing limit cycles.

\noindent\textbf{Case $(\mathrm{Ls})$–$(\mathtt{Q}_1)$.}
We exhibit a discontinuous piecewise differential system of  type $(\mathtt{L}_s)$-$(\mathtt{Q}_1)$
possessing four crossing limit cycles.

In $\Sigma^+$, we consider the linear Hamiltonian saddle

\begin{align}\label{sist1-Hs-H1}
\dot{x} =&  \frac{79}{100}x - \frac{1}{20}y-\frac{53}{100} ,\\
\dot{y} =&  - \frac{33}{100}x - \frac{79}{100}y + \frac{22}{25}\nonumber, 
 \end{align} with the first integral
\begin{equation*}
    H_s(x,y) =\frac{33}{200}x^2 + \frac{79}{100}\,x\,y - \frac{1}{40}y^2 - \frac{22}{25}x - \frac{53}{100}y
\end{equation*}
In $\Sigma^{-}$, we consider the quadratic isochronous differential center of type $(\mathtt{Q}_1)$
\begin{equation}\label{sist2-Hs-H1}
    \begin{array}{ll}
    \dot{x} =&0.386769 + 6.10786x - 2.58171x^2 - 4.01571y + 0.411844xy + 0.743863y^2,\\  [4pt]
    \dot{y} =& -0.637351 + 10.6644x - 0.546859x^2 - 5.71494y - 3.9509xy + 2.39596y^2, \\
 \end{array}
\end{equation}
with first integral    
\begin{equation*}
\begin{array}{ll}
H_1(x,y) =& \frac{(0.0982285 - 2.27858x + 1.30094y)^2
+ (0.903598 - 0.14395x - \frac{13}{25}y)^2  }
{1 + 2(0.903598 - 0.14395x - \frac{13}{25}y)}.
\end{array}
\end{equation*}
By solving system \eqref{crossingEq}, we obtain four pairs of real solutions $(p_i,q_i)$, with $p_i = (x_i,0)$ and $q_i = (0,y_i)$ for $i=1,\dots,4$, satisfying $x_i,y_i>0$:
\begin{align*}
p_1 \approx & (0.247078, 0), \quad q_1 \approx (0, 0.384272), \\
p_2 \approx & (0.415771, 0), \quad q_2 \approx (0, 0.618477), \\
p_3 \approx & (0.612878, 0), \quad q_3 \approx (0, 0.865349), \\
p_4 \approx & (0.688027, 0), \quad q_4 \approx (0, 0.952239).
\end{align*}
Then, the discontinuous piecewise
linear differential system \eqref{sist1-Hs-H1}–\eqref{sist2-Hs-H1} gives rise to the
four limit cycles showed in Figure \ref{fig-Hs-H1}. 
\begin{figure}[h] 
\begin{center}
\includegraphics[scale=0.3]{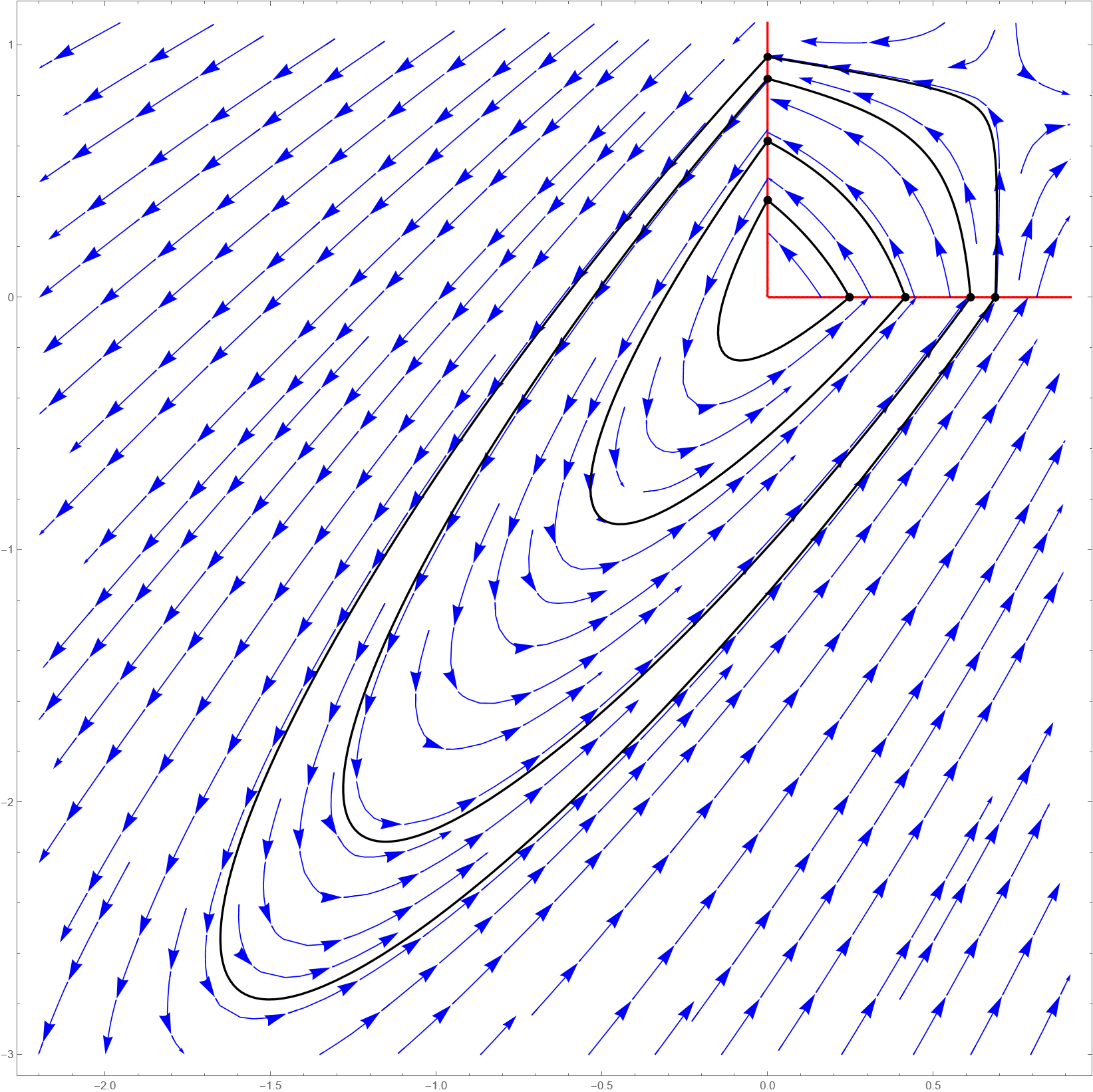}
\caption{The four limit cycle of the discontinuous piecewise differential system \eqref{sist1-Hs-H1}-\eqref{sist2-Hs-H1}  of Theorem \ref{Teor-Principal}.}
\label{fig-Hs-H1}
\end{center}
\end{figure}
\smallskip

\noindent\textbf{Case $(\mathrm{Ls})$–$(\mathtt{Q}_2)$.}
We exhibit a system of type $(\mathtt{L}_s)$-$(\mathtt{Q}_2)$ possessing four crossing limit cycles

In $\Sigma^+$, we consider the linear Hamiltonian saddle
\begin{align}\label{sist1-Hs-H2}
    \dot{x} &=- \frac{47}{100} x + \frac{19}{100}y+\frac{11}{100} , \\
    \dot{y} &=  - \frac{19}{25}x + \frac{47}{100}y-\frac{17}{20}\nonumber.
\end{align}
with the first integral
\begin{equation*}
    H_s(x,y) =   \frac{19}{50}x^2 - \frac{47}{100}\,x\,y + \frac{19}{200}y^2+\frac{17}{20}x+ \frac{11}{100}y .
\end{equation*}
In $\Sigma^{-}$, we consider the quadratic isochronous differential center of type $(\mathtt{Q}_2)$
\begin{equation}\label{sist2-Hs-H2}
    \begin{array}{ll}
\dot{x} &= 0.298027 - 0.325852 x + 0.064 x^2 + 0.485269 y - 0.0305115 x y - 8.47033 \cdot 10^{-19} y^2,\\
\dot{y} &= -2.27794 - 2.17886 x + 0.435481 y + 0.064 x y - 0.0305115 y^2,
 \end{array}
\end{equation}

with first integral    
\begin{equation*}
    \begin{array}{ll}
H_2(x,y) =&\frac{\left(0.0365432 + \frac{8}{125}x - 0.0305115\,y\right)^2 
+ \left(-0.0872418 - 0.0409482\,x - \frac{1}{50}y\right)^2}
{\left(0.912758 - 0.0409482\,x - \frac{1}{50}y\right)^2}.
    \end{array}
\end{equation*}

By solving system \eqref{crossingEq}, we obtain four pairs of real solutions $(p_i,q_i)$, with $p_i = (x_i,0)$ and $q_i = (0,y_i)$ for $i=1,\dots,4$, satisfying $x_i,y_i>0$:
\begin{align*}
p_1 \approx & (0.879937, 0), \quad q_1 \approx (0, 2.78341), \\
p_2 \approx & (0.963775, 0), \quad q_2 \approx (0, 2.98109), \\
p_3 \approx & (1.01989, 0), \quad q_3 \approx (0, 3.11175), \\
p_4 \approx & (1.06254, 0), \quad q_4 \approx (0, 3.21026).
\end{align*}

Then, the discontinuous piecewise linear differential system \eqref{sist1-Hs-H2}–\eqref{sist2-Hs-H2} gives rise to the
four limit cycles showed in Figure \ref{fig-Hs-H2}. 
\begin{figure}[h] 
\begin{center}
\includegraphics[scale=0.3]{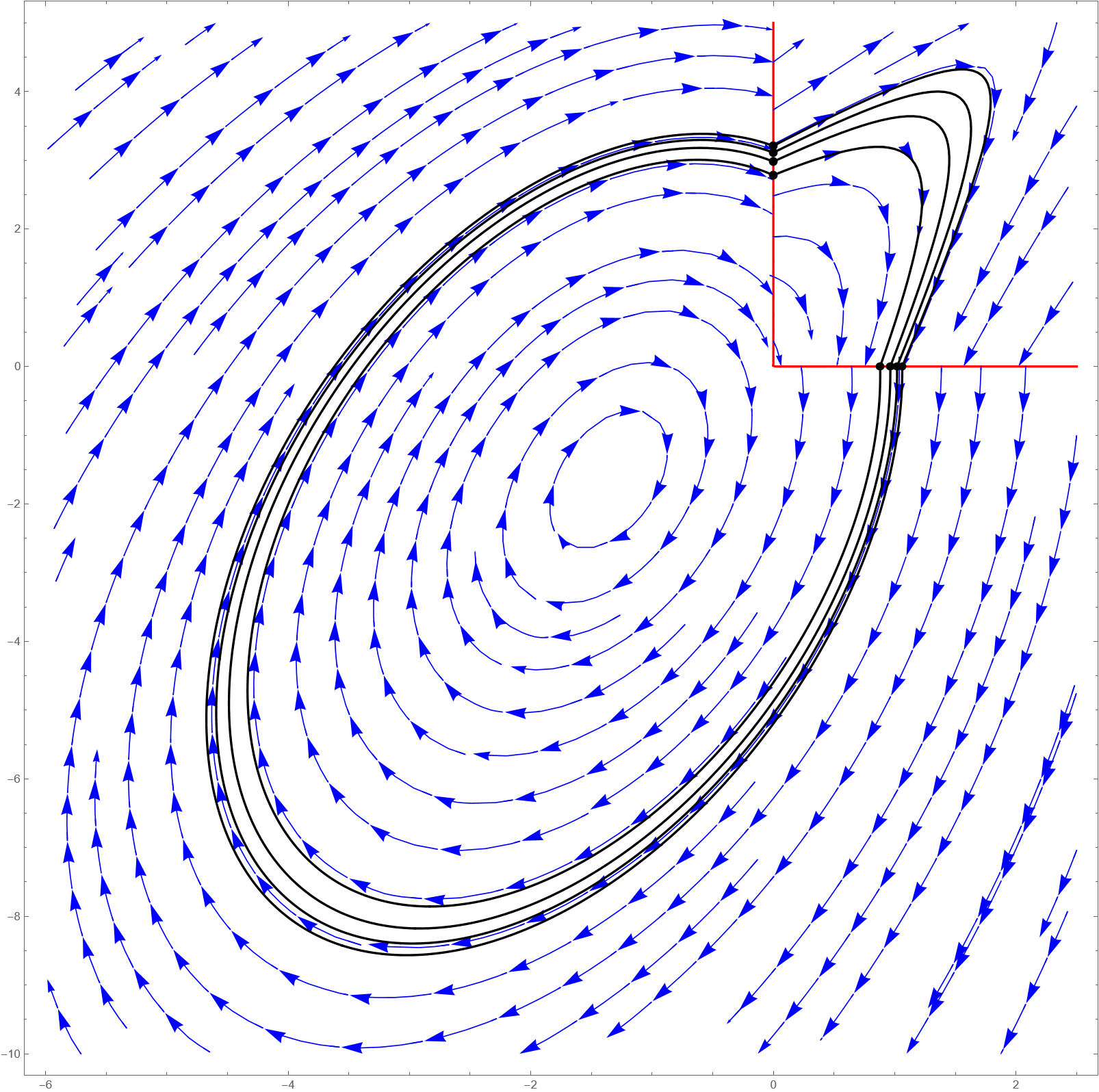}
\caption{The four limit cycle of the discontinuous piecewise differential system \eqref{sist1-Hs-H2}-\eqref{sist2-Hs-H2}  of Theorem \ref{Teor-Principal}.}
\label{fig-Hs-H2}
\end{center}
\end{figure}

\smallskip

\noindent\textbf{Case $(\mathrm{Ls})$–$(\mathtt{Q}_3)$.}
We exhibit a system of type $(\mathtt{L}_s)$-$(\mathtt{Q}_3)$ possessing four crossing limit cycles

In $\Sigma^+$, we consider the linear Hamiltonian saddle
 \begin{align}\label{sist1-Hs-H3}
    \dot{x} &= \frac{9}{10}\,x - \frac{4}{5}\,y - \frac{1}{10}, \\
    \dot{y} &= - \frac{1}{5}x + \frac{9}{10}y + \frac{3}{5}\nonumber.
\end{align}
 with the first integral
\begin{equation*}
    H_s(x,y) = \frac{1}{10}\,x^2 - \frac{9}{10}\,x\,y  - \frac{2}{5}\,y^2 - \frac{3}{5}\,x + \frac{1}{10}\,y  .
\end{equation*}
In $\Sigma^{-}$, we consider the quadratic isochronous differential center of type $(\mathtt{Q}_3)$
\begin{equation}\label{sist2-Hs-H3}
    \begin{array}{ll}
    \dot{x} =& \big( 1.46492 + 6.83181 x + 5.7871 x^2 - 4.13813 y - 6.56629 x y + 1.64208 y^2 \big), \\[10pt]
    \dot{y} =& \big( -0.0942582 + 6.10742 x + 6.61795 x^2 - 3.75871 y - 7.12975 x y + 1.58864 y^2 \big).
    \end{array}
\end{equation}
with first integral  
\begin{equation*}
\begin{aligned}
H_3(x,y) =\;& \frac{1}{-3 + 16\left(0.657025 + 0.639096 x - \frac{3}{11}y\right)} \Big(
9\Big(\left(0.657025 + 0.639096 x - \tfrac{3}{11}y\right)^2 \\
&\quad + \left(-0.384136 - \tfrac{5}{9}x + 0.423626 y\right)^2\Big) \\
&\quad - 24\left(0.657025 + 0.639096 x - \tfrac{3}{11}y\right)
\left(-0.384136 - \tfrac{5}{9}x + 0.423626 y\right)^2 \\
&\quad + 16\left(-0.384136 - \tfrac{5}{9}x + 0.423626 y\right)^4
\Big).
\end{aligned}
\end{equation*}

By solving system \eqref{crossingEq}, we obtain four pairs of real solutions $(p_i,q_i)$, with $p_i = (x_i,0)$ and $q_i = (0,y_i)$ for $i=1,\dots,4$, satisfying $x_i,y_i>0$:
\begin{align*}
p_1 \approx & (2.21003, 0), \quad q_1 \approx (0, 1.57745), \\
p_2 \approx & (1.64943, 0), \quad q_2 \approx (0, 1.47022), \\
p_3 \approx & (1.34953, 0), \quad q_3 \approx (0, 1.38381), \\
p_4 \approx & (1.07022, 0), \quad q_4 \approx (0, 1.28025).
\end{align*}

 Then, the discontinuous piecewise linear differential system \eqref{sist1-Hs-H3}–\eqref{sist2-Hs-H3} gives rise to the four limit cycles showed in Figure \ref{fig-Hs-H3}. 
\begin{figure}[h] 
\begin{center}
\includegraphics[scale=0.3]{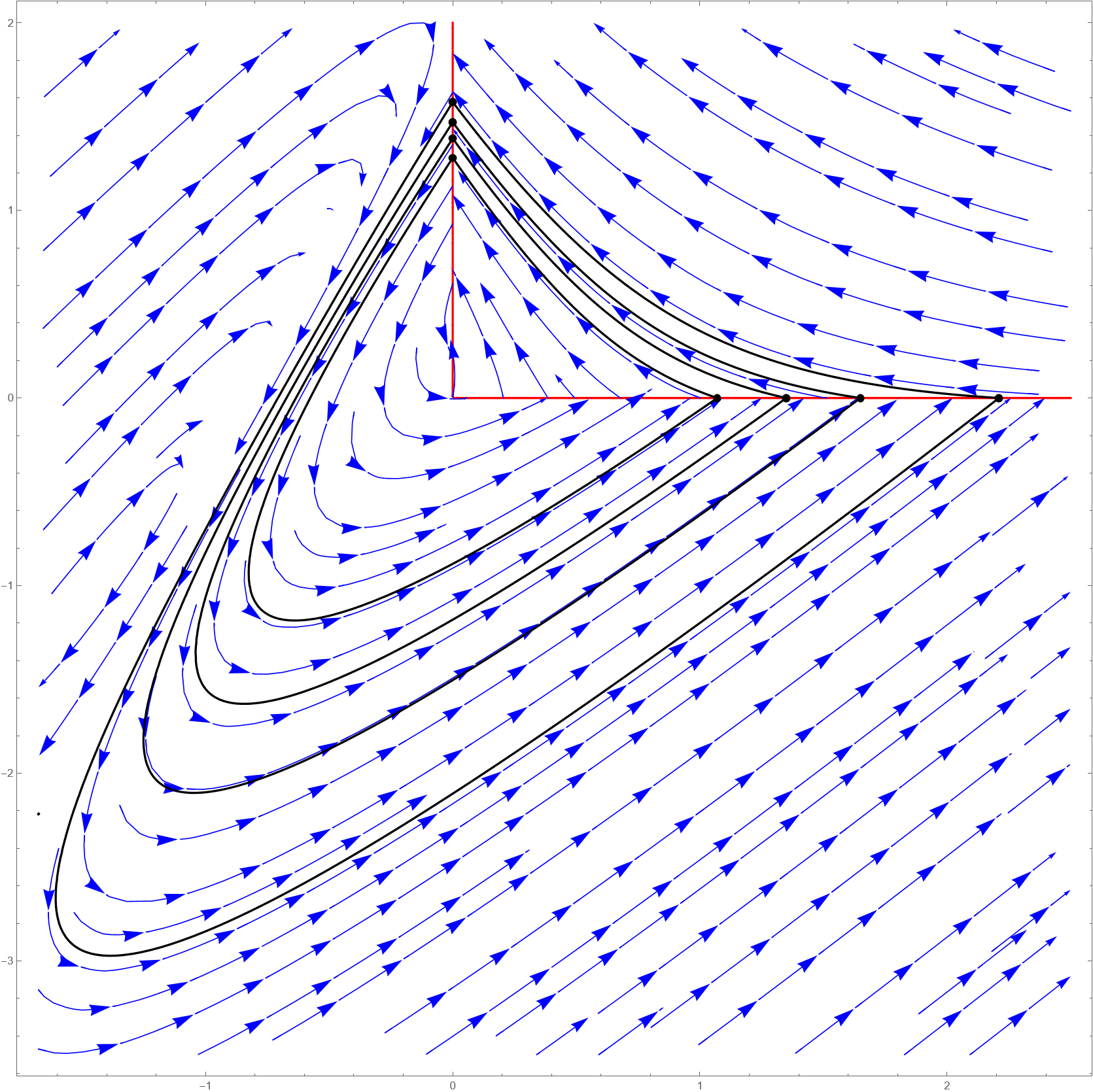}
\caption{The four limit cycle of the discontinuous piecewise differential system \eqref{sist1-Hs-H3}-\eqref{sist2-Hs-H3}  of Theorem \ref{Teor-Principal}.}
\label{fig-Hs-H3}
\end{center}
\end{figure}

\smallskip

\noindent\textbf{Case $(\mathrm{Ls})$–$(\mathtt{Q}_4)$.}
We exhibit a system of type $(\mathtt{L}_s)$-$(\mathtt{Q}_4)$ possessing four crossing limit cycles

In $\Sigma^+$, we consider the linear Hamiltonian saddle
\begin{align}\label{sist1-Hs-H4}
    \dot{x} = & - \frac{12}{25}\,x + 0.00914158\,y -\frac{3}{100} \\
    \dot{y} = & 0.0193132\,x + \frac{12}{25}\,y + 0.0309612, 
 \end{align} with the first integral
\begin{equation*}
    H_s(x,y) = - 0.00965661\,x^2 - \frac{12}{25}\,x\,y + 0.00457079\,y^2 - 0.0309612\,x  - \frac{3}{100}\,y  .
\end{equation*}
In $\Sigma^{-}$, we consider the quadratic isochronous differential center of type
$(\mathtt{Q}_4)$
\begin{equation}\label{sist2-Hs-H4}
    \begin{array}{ll}
    \dot{x} =& \frac{8447}{337500} - \frac{6899}{37500}x + \frac{6329}{3375}x^2 - \frac{28429}{37500}y - \frac{9752}{3375}xy + \frac{119}{3375}y^2, \\[8pt]
    \dot{y} =& -\frac{14657}{337500} + \frac{42469}{37500}x - \frac{1559}{3375}x^2 + \frac{633}{12500}y + \frac{11192}{3375}xy - \frac{3449}{3375}y^2,
 \end{array}
\end{equation}
with first integral    
\begin{equation*}
    \begin{array}{ll}
H_4(x,y) =& \frac{
9\Big( \left(-\frac{1}{100} + \frac{53}{100}x - \frac{37}{100}y\right)^2
+ \left(-\frac{3}{100} + \frac{x}{2} + \frac{y}{2}\right)^2 \Big)
+ 24\left(-\frac{3}{100} + \frac{x}{2} + \frac{y}{2}\right)^3
+ 16\left(-\frac{3}{100} + \frac{x}{2} + \frac{y}{2}\right)^4
}{
\left(3 + 8\left(-\frac{3}{100} + \frac{x}{2} + \frac{y}{2}\right)\right)^4
}.
\end{array}
\end{equation*}
By solving system \eqref{crossingEq}, we obtain four pairs of real solutions $(p_i,q_i)$, with $p_i = (x_i,0)$ and $q_i = (0,y_i)$ for $i=1,\dots,4$, satisfying $x_i,y_i>0$:
\begin{align*}
p_1 \approx & (1.15329, 0), \quad q_1 \approx (0, 2.89772), \\
p_2 \approx & (0.874219, 0), \quad q_2 \approx (0, 1.48358), \\
p_3 \approx & (0.491262, 0), \quad q_3 \approx (0, 0.648826), \\
p_4 \approx & (0.300288, 0), \quad q_4 \approx (0, 0.358519).
\end{align*}

Then, the discontinuous piecewise linear differential system
\eqref{sist1-Hs-H4}–\eqref{sist2-Hs-H4} gives rise to the
four limit cycles showed in Figure  \ref{fig-Hs-H4}. 
\begin{figure}[h] 
\begin{center}
\includegraphics[scale=0.3]{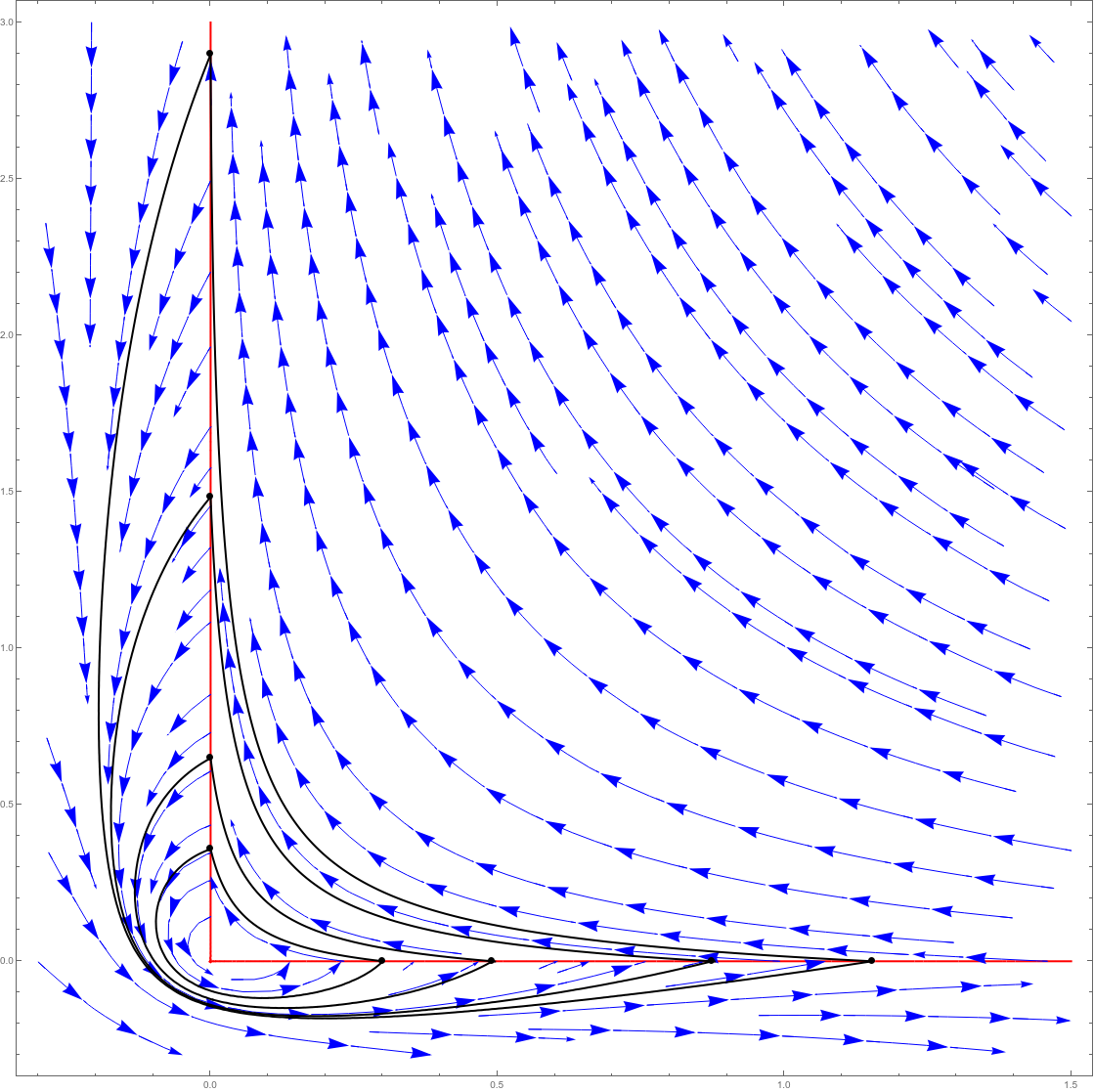}
\caption{The four limit cycle of the discontinuous piecewise differential system \eqref{sist1-Hs-H4}-\eqref{sist2-Hs-H4}  of Theorem \ref{Teor-Principal}.}
\label{fig-Hs-H4}
\end{center}
\end{figure}
\end{proof}
\smallskip

The attainability of the upper bound established in Theorem~\ref{Teor-Principal} for the maximum number of limit cycles in this class of discontinuous piecewise differential systems is still unknown.

 \section*{Acknowledgments}
 The first author was partially supported by the 2026 Summer Postdoctoral Program of the Instituto de Matemática Pura e Aplicada (IMPA).
 The second author was partially supported by  CNPq 169201/2023-6  grant.

\appendix
\label{appendixA}
\section{ }
For completeness, we provide the explicit expressions of the polynomials $P_i(x,y)$, $i=1,2,3,4$, appearing in the proof of Theorem~\ref{Teor-Principal}. These polynomials are obtained from the equations
\[
H_i(x,0)-H_i(0,y)=0,
\]
after clearing denominators.

\begin{equation*}
 \begin{aligned}
P_1(x,y)=\;&
2\bigl(a_1^2+\alpha_1^2\bigr)\beta_1\,x^2y
-2\alpha_1\bigl(b_1^2+\beta_1^2\bigr)\,xy^2
-4c_1\bigl(b_1\alpha_1-a_1\beta_1\bigr)\,xy+\bigl(a_1^2+\alpha_1^2\bigr)\\
&
\bigl(1+2\gamma_1\bigr)x^2-\bigl(b_1^2+\beta_1^2\bigr)\bigl(1+2\gamma_1\bigr)y^2+2\Bigl(a_1c_1-c_1^2\alpha_1
+2a_1c_1\gamma_1+\alpha_1\gamma_1+\alpha_1\gamma_1^2\Bigr)x\\
&
-2\Bigl(b_1c_1-c_1^2\beta_1 + 2b_1c_1\gamma_1+\beta_1\gamma_1+\beta_1\gamma_1^2\Bigr)y,
\end{aligned}
\end{equation*}

\begin{equation*}
 \begin{aligned}
P_2(x,y)=\;&
\bigl(a_{1}^{2}\beta_{1}^{2}-b_{1}^{2}\alpha_{1}^{2}\bigr)x^{2}y^{2}
-4c_{1}(1+\gamma_{1})\bigl(b_{1}\alpha_{1}-a_{1}\beta_{1}\bigr)xy-2\Bigl(
b_{1}c_{1}\alpha_{1}^{2}
-a_{1}^{2}\beta_{1}
-\alpha_{1}^{2}\beta_{1}
\\
&-a_{1}^{2}\beta_{1}\gamma_{1}
\Bigr)x^{2}y -\Bigl(
b_{1}^{2}\alpha_{1}
-a_{1}c_{1}\beta_{1}^{2}
+\alpha_{1}\beta_{1}^{2}
+b_{1}^{2}\alpha_{1}\gamma_{1}
\Bigr)2xy^{2}+\Bigl(
a_{1}^{2}+\alpha_{1}^{2}
-c_{1}^{2}\alpha_{1}^{2}
\\
& +2(a_{1}^{2}+\alpha_{1}^{2})\gamma_{1}
+a_{1}^{2}\gamma_{1}^{2}
\Bigr)x^{2}+\Bigl(
-b_{1}^{2}-\beta_{1}^{2}
+c_{1}^{2}\beta_{1}^{2}
-2(b_{1}^{2}+\beta_{1}^{2})\gamma_{1}
-b_{1}^{2}\gamma_{1}^{2}
\Bigr)y^{2}\\
&+2(1+\gamma_{1})
\bigl(a_{1}c_{1}-c_{1}^{2}\alpha_{1}
+a_{1}c_{1}\gamma_{1}
+\alpha_{1}\gamma_{1}\bigr)x-2(1+\gamma_{1})
\bigl(b_{1}c_{1}-c_{1}^{2}\beta_{1}
+b_{1}c_{1}\gamma_{1}
+\beta_{1}\gamma_{1}\bigr)y,
\end{aligned}
\end{equation*}

\begin{equation*}
 \begin{aligned}
P_3(x,y)=\;&
-256\,b_1^{4}\alpha_1\,x y^{4}
+256\,a_1^{4}\beta_1\,x^{4}y
+128\,\bigl(8a_1^{3}c_1\beta_1-3a_1^{2}\alpha_1\beta_1\bigr)x^{3}y-128\,\bigl(8b_1^{3}c_1\alpha_1\\
&
-3b_1^{2}\alpha_1\beta_1\bigr)x y^{3}-32c_1\bigl(b_1\alpha_1-a_1\beta_1\bigr)
\bigl(9+32c_1^{2}-24\gamma_1\bigr)xy + 8a_1^{2}\bigl(8a_1c_1-3\alpha_1\bigr)
\\
&
\bigl(-3+16\gamma_1\bigr)x^{3}-8b_1^{2}\bigl(8b_1c_1-3\beta_1\bigr)
\bigl(-3+16\gamma_1\bigr)y^{3}-
3\bigl(-3+16\gamma_1\bigr)
\bigl(-3a_1^{2}-32a_1^{2}c_1^{2}
\\
& +16a_1c_1\alpha_1-3\alpha_1^{2}
+8a_1^{2}\gamma_1\bigr)x^{2}+16\bigl(-3a_1^{4}+16a_1^{4}\gamma_1\bigr)x^{4}
+3\bigl(-3+16\gamma_1\bigr)
\bigl(-3b_1^{2}\\
&-32b_1^{2}c_1^{2}
+16b_1c_1\beta_1-3\beta_1^{2}
+8b_1^{2}\gamma_1\bigr)y^{2}-16\bigl(-3b_1^{4}+16b_1^{4}\gamma_1\bigr)y^{4}
-2\bigl(9+32c_1^{2}\\
&-24\gamma_1\bigr)
\bigl(3a_1c_1+4c_1^{2}\alpha_1
-16a_1c_1\gamma_1
+3\alpha_1\gamma_1\bigr)x-48\bigl(3b_1^{2}\alpha_1
+32b_1^{2}c_1^{2}\alpha_1
-16b_1c_1\alpha_1\beta_1\\
&+3\alpha_1\beta_1^{2}
-8b_1^{2}\alpha_1\gamma_1\bigr)xy^{2}
+2\bigl(9+32c_1^{2}-24\gamma_1\bigr)
\bigl(3b_1c_1+4c_1^{2}\beta_1
-16b_1c_1\gamma_1
+3\beta_1\gamma_1\bigr)y\\ &+48\bigl(3a_1^{2}\beta_1
+32a_1^{2}c_1^{2}\beta_1
-16a_1c_1\alpha_1\beta_1
+3\alpha_1^{2}\beta_1
-8a_1^{2}\beta_1\gamma_1\bigr)x^{2}y,
\end{aligned}
\end{equation*}

\begin{equation*}
 \begin{aligned}
P_4(x,y)=\;&
18432 (2 a_{1}^{2}-\alpha_{1}^{2}) \beta_{1}^{4}  x^{2} y^{4} 
-18432 \alpha_{1}^{4} (2 b_{1}^{2}-\beta_{1}^{2})x^{4} y^{2} +
9216 (2 a_{1}^{2}-\alpha_{1}^{2}) \beta_{1}^{3} (3+8\gamma_{1}) x^{2} y^{3}
\\
&-9216 \alpha_{1}^{3} (2 b_{1}^{2}-\beta_{1}^{2}) (3+8\gamma_{1})x^{3} y^{2} -3456  (b_{1}\alpha_{1}-a_{1}\beta_{1})(b_{1}\alpha_{1}+a_{1}\beta_{1})(3+8\gamma_{1})^{2}x^{2} y^{2}
 \\
&-576 c_{1} (b_{1}\alpha_{1}-a_{1}\beta_{1})(3+8\gamma_{1})^{3} x y +4608  \beta_{1}^{4} (16 a_{1} c_{1}-3\alpha_{1}-8\alpha_{1}\gamma_{1})x y^{4}
\\
&-2304 \beta_{1}^{3} (3+8\gamma_{1})(-16 a_{1} c_{1}+3\alpha_{1}+8\alpha_{1}\gamma_{1})x y^{3}  -4608 \alpha_{1}^{4} (16 b_{1} c_{1}-3\beta_{1}-8\beta_{1}\gamma_{1})x^{4} y 
\\
&+2304  \alpha_{1}^{3} (3+8\gamma_{1})(-16 b_{1} c_{1}+3\beta_{1}+8\beta_{1}\gamma_{1})x^{3} y +288  (3+8\gamma_{1})^{2}
(-24 b_{1} c_{1}\alpha_{1}^{2}+3 a_{1}^{2}\beta_{1}\\
&
+3\alpha_{1}^{2}\beta_{1}+8 a_{1}^{2}\beta_{1}\gamma_{1}
+8\alpha_{1}^{2}\beta_{1}\gamma_{1})x^{2} y -288 (3+8\gamma_{1})^{2}
(3 b_{1}^{2}\alpha_{1}-24 a_{1} c_{1}\beta_{1}^{2}
+3\alpha_{1}\beta_{1}^{2} \\
&+8 b_{1}^{2}\alpha_{1}\gamma_{1}
+8\alpha_{1}\beta_{1}^{2}\gamma_{1})x y^{2} -144 \alpha_{1}^{4}(-9+256 c_{1}^{2}-96\gamma_{1}-128\gamma_{1}^{2})x^{4}
+144 \beta_{1}^{4}(-9+256 c_{1}^{2} \\
&-96\gamma_{1}-128\gamma_{1}^{2})y^{4}+72 \alpha_{1}^{3}(3+8\gamma_{1})(9-256 c_{1}^{2}+96\gamma_{1}  +128\gamma_{1}^{2})x^{3}
-72 \beta_{1}^{3}(3+8\gamma_{1})\\
&(9-256 c_{1}^{2}+96\gamma_{1}+128\gamma_{1}^{2})y^{3}+18 (3+8\gamma_{1})^{3}
(3 a_{1} c_{1}-16 c_{1}^{2}\alpha_{1}
+8 a_{1} c_{1}\gamma_{1}+3\alpha_{1}\gamma_{1}
  \\
&+4\alpha_{1}\gamma_{1}^{2})x+9(3+8\gamma_{1})^{2}
(9 a_{1}^{2}+9\alpha_{1}^{2}-384 c_{1}^{2}\alpha_{1}^{2}
+48 a_{1}^{2}\gamma_{1}+120\alpha_{1}^{2}\gamma_{1}
+64 a_{1}^{2}\gamma_{1}^{2}  \\
&+160\alpha_{1}^{2}\gamma_{1}^{2}) x^{2}-18  (3+8\gamma_{1})^{3}
\left(3 b_{1} c_{1}-16 c_{1}^{2}\beta_{1}
+8 b_{1} c_{1}\gamma_{1}+3\beta_{1}\gamma_{1}
+4\beta_{1}\gamma_{1}^{2}\right)y \\
&-9 (3+8\gamma_{1})^{2}
\left(9 b_{1}^{2}+9\beta_{1}^{2}-384 c_{1}^{2}\beta_{1}^{2}
+48 b_{1}^{2}\gamma_{1}+120\beta_{1}^{2}\gamma_{1}
+64 b_{1}^{2}\gamma_{1}^{2}+160\beta_{1}^{2}\gamma_{1}^{2}\right)y^{2}.
\end{aligned}
\end{equation*}

\bibliographystyle{acm}
\bibliography{sample.bib}



\end{document}